\documentclass[a4paper,11pt]{article}
 
\usepackage{latexsym,amssymb,amsmath,amsthm,mathrsfs}  
\usepackage[width=14.5cm]{geometry}

\newtheorem{theorem}{Theorem}
\newtheorem{corollary}[theorem]{Corollary}
\newtheorem{definition}[theorem]{Definition}

\newtheorem{example}[theorem]{Example}
\newtheorem{remark}[theorem]{Remark}

\begin{document}

\title{A Cauchy--Riemann-Type CK Extension for Monogenic Functions}

\author{Dixan Pe\~na Pe\~na\\
\small{e-mail: dixanpena@gmail.com}}

\date{\small{Department of Mathematics: Analysis, Logic and Discrete Mathematics\\Ghent University, Krijgslaan 297\\9000 Gent, Belgium}}

\maketitle

\begin{abstract}
\noindent We introduce a generalized Cauchy--Kowalevski extension associated with a biaxial decomposition of $\mathbb{R}^{m+1}$, allowing monogenic functions to be reconstructed from data prescribed on the subspace $\{X\in\mathbb{R}^{m+1}: z=0\}$ where $z=x_0+x_1e_1$. The extension admits a natural decomposition into monogenic layers, leading to operators $\mathrm{CK}_{z,k}$. Classes of biaxial monogenic functions generated by these operators are characterized by systems of partial differential equations, and explicit power series solutions are obtained. The action of the operators $\mathrm{CK}_{z,k}$ on homogeneous polynomials is then investigated, leading to a class of biaxial monogenic plane waves and the associated system of partial differential equations.\vspace{0.1cm}\\
\textit{Mathematics Subject Classification}: 30G20, 30G35, 33C10.
\end{abstract}

\section{Introduction and preliminaries}

Let $\mathbb{R}_{0,m}$ ($m \in \mathbb{N}$) denote the real Clifford algebra of signature $(0,m)$ generated by the elements $e_1, \ldots, e_m$ (see e.g. \cite{Cl,Lou}). These generators can be identified with the standard Euclidean basis of $\mathbb{R}^m$. The multiplication in $\mathbb{R}_{0,m}$ is defined so that, for any vector $\underline{x} = \sum_{j=1}^m x_j e_j \in \mathbb{R}^m$, one has
\[
\underline{x}^2 = -|\underline{x}|^2 = -\sum_{j=1}^m x_j^2.
\]
This condition implies that the generators $e_j$ satisfy the relations
\begin{alignat*}{2}
e_j^2 &= -1, &\qquad & j = 1, \ldots, m, \\
e_j e_k + e_k e_j &= 0, &\qquad & 1 \leq j \neq k \leq m.
\end{alignat*}
These rules completely determine the algebraic structure of $\mathbb{R}_{0,m}$. A general element $a \in \mathbb{R}_{0,m}$ can be expressed as
\[
a = \sum_{A} a_A e_A, \qquad a_A \in \mathbb{R},
\]
where the basis elements $e_A = e_{j_1} \cdots e_{j_k}$ correspond to subsets $A = \{j_1, \ldots, j_k\} \subset \{1, \ldots, m\}$ with $j_1 < \cdots < j_k$, and where we adopt the convention $e_{\emptyset} = 1$. Therefore, $\mathbb{R}_{0,m}$ has dimension $2^m$.

The conjugation is defined as an anti-involution on $\mathbb{R}_{0,m}$ by
\[
\overline{a} = \sum_A a_A \, \overline{e_A},
\]
where the conjugation of basis elements is given by
\[
\overline{e_A} = \overline{e_{j_k}} \cdots \overline{e_{j_1}}, \qquad
\overline{e_j} = -e_j, \quad j = 1, \ldots, m.
\]
Observe that $\mathbb{R}^{m+1}$ can be naturally embedded into the real Clifford algebra $\mathbb{R}_{0,m}$ by identifying each point $(x_0,x_1,\ldots,x_m) \in \mathbb{R}^{m+1}$ with the paravector
\[
X = x_0 + \underline{x} = x_0 + \sum_{j=1}^m x_j e_j.
\]
Clifford analysis is a higher-dimensional generalization of complex function theory based on Clifford algebras and Dirac-type operators (see e.g. \cite{BDSom,Del, DSomS, GiMu, GuSp}). Its central object is the generalized Cauchy--Riemann operator
\[\partial_X = \partial_{x_0} + \partial_{\underline{x}},\]
where
\[
\partial_{\underline{x}} = \sum_{j=1}^m e_j \partial_{x_j}
\]
is the Dirac operator in $\mathbb{R}^m$.

\begin{definition}
Let $\Omega$ be an open subset of $\mathbb{R}^{m+1}$. A continuously differentiable function $f : \Omega \to \mathbb{R}_{0,m}$ is called left monogenic, or simply monogenic, if it satisfies
\[
\partial_X f(X) = 0 \quad \text{in } \Omega.
\]
\end{definition}

The operator $\partial_X$ provides a factorization of the Laplacian in $\mathbb{R}^{m+1}$, namely,
\[
\Delta_{m+1} = \sum_{j=0}^m \partial_{x_j}^2 
= \partial_X \, \overline{\partial_X} = \overline{\partial_X} \, \partial_X.
\]

For a differentiable $\mathbb{R}$-valued function $\phi$ and a differentiable $\mathbb{R}_{0,m}$-valued function $g$, we have
\begin{equation}\label{LeibnizRuleOne}
\partial_{\underline{x}}(\phi\,g) = (\partial_{\underline{x}} \phi) g + \phi (\partial_{\underline{x}} g).
\end{equation}

Moreover, for a differentiable vector-valued function $\underline{f} = \sum_{j=1}^m f_j e_j$, we have
\[
\partial_{\underline{x}}(\underline{f} g) = (\partial_{\underline{x}} \underline{f}) g - \underline{f} (\partial_{\underline{x}} g) - 2 \sum_{j=1}^m f_j (\partial_{x_j} g).
\]

Let
\[
\beta_{j,n} =
\begin{cases}
j, & \text{if $j$ is even}, \\
2n + m + j - 1, & \text{if $j$ is odd},
\end{cases}
\]
and let $P_n(\underline{x})$ denote a homogeneous monogenic polynomial in $\mathbb{R}^m$ of degree $n$. Using the above Leibniz rules, together with Euler's theorem for homogeneous functions, we obtain the following identity
\begin{equation}\label{LeibnizRuleTwo}
\partial_{\underline{x}}\big(\underline{x}^j P_n(\underline{x})\big)
=
- \beta_{j,n}\, \underline{x}^{j-1} P_n(\underline{x}), \quad j \ge 1.
\end{equation}

A fundamental tool in Clifford analysis is the Cauchy--Kowalevski (CK) extension, which provides a natural way to extend functions defined on $\mathbb{R}^m$ to monogenic functions in $\mathbb{R}^{m+1}$ (see e.g. \cite{DSomS,SomExp,SomAx}). Let $g$ be a real-analytic function defined on an open subset of $\mathbb{R}^m$. Then there exists a unique monogenic extension $f$ in $\mathbb{R}^{m+1}$ satisfying
\[
f(0,\underline{x}) = g(\underline{x}),
\]
denoted by $\mathrm{CK}[g(\underline{x})](X)$ and given explicitly by
\[
f(X) = \sum_{k=0}^{\infty} \frac{(-x_0)^k}{k!}\, \partial_{\underline{x}}^k g(\underline{x}).
\]
This series admits a compact representation in terms of functional calculus, namely
\[
f(X) = \exp(-x_0 \partial_{\underline{x}})\, g(\underline{x}).
\]
The CK extension plays a central role in the construction of important families of monogenic functions. Two notable examples are axial monogenic functions and monogenic plane waves.

An axial monogenic function of degree $n$ has the form
\[
\left(A(x_0,r)+\frac{\underline x}{r}\,B(x_0,r)\right)P_n(\underline x),\qquad r=|\underline{x}|,
\]
where $P_n(\underline{x})$ is as above, and the $\mathbb{R}$-valued functions $A$ and $B$ satisfy the Vekua-type system
\begin{align*}
\partial_{x_0}A-\partial_rB&=\displaystyle{\frac{2n+m-1}{r}}\,B\\
\partial_{x_0}B+\partial_rA&=0.
\end{align*}
Monogenic plane waves are given by 
\[M(u,v)-\frac{\underline{t}}{|\underline{t}|}N(u,v), \qquad u=\langle \underline{x},\underline{t}\rangle, \qquad v=x_0|\underline{t}|, \]
where $\underline{t}\in\mathbb{R}^m$ is fixed, and the $\mathbb{R}$-valued functions $M$ and $N$ satisfy the Cauchy--Riemann equations in the variables $u$ and $v$.

These examples illustrate how the CK extension naturally gives rise to special classes of monogenic functions characterized by lower-dimensional systems of partial differential equations. Motivated by this observation, as well as by a power series expansion introduced in \cite{DPFS}, we develop in this paper a CK-type extension adapted to the decomposition
\[\mathbb{R}^{m+1}=\mathbb{R}^2\oplus\mathbb{R}^{m-1}, \]
where the two-dimensional component is described by the variable $z=x_0+x_1e_1$. As in the classical setting, this construction gives rise to families of monogenic functions and associated partial differential systems.

The paper is organized as follows. 

In Section~2, we introduce a generalized CK extension in which monogenic functions are reconstructed from data prescribed on the codimension-2 subspace $\{X\in\mathbb{R}^{m+1}: z=0\}$. We show that the extension naturally decomposes into monogenic layers and gives rise to operators $\mathrm{CK}_{z,k}$ admitting a representation in terms of modified Bessel functions of the first kind.

In Section~3, we investigate classes of biaxial monogenic functions generated by the operators $\mathrm{CK}_{z,k}$. We derive the corresponding systems of partial differential equations, determine their power series solutions, and obtain explicit hypergeometric representations for solutions arising from power initial conditions.

Finally, in Section~4, we investigate the action of the operators $\mathrm{CK}_{z,k}$ on homogeneous polynomials, which leads naturally to the introduction of a class of biaxial monogenic plane waves and the derivation of the associated partial differential system.

\section{Biaxial decomposition and a generalized CK extension}

We consider the biaxial decomposition $\mathbb{R}^{m+1} = \mathbb{R}^2 \oplus \mathbb{R}^{m-1}$, which allows us to represent any vector $X \in \mathbb{R}^{m+1}$ in the form
\[
X = z + \underline{y},
\]
where
\[
z = x_0 + x_1 e_1, \qquad \underline{y} = \sum_{j=2}^m x_j e_j.
\]
With respect to this decomposition, the generalized Cauchy--Riemann operator $\partial_X$ splits as
\[
\partial_X = 2 \partial_{\bar z} + \partial_{\underline{y}},
\]
where 
\[
\partial_{\bar z} = \tfrac{1}{2}(\partial_{x_0} + e_1 \partial_{x_1})
\]
is the analogue of the classical Cauchy--Riemann operator in the $(x_0,x_1)$-plane, and 
\[
\partial_{\underline{y}} = \sum_{j=2}^m e_j \partial_{x_j}.
\]
We also introduce the operator
\[
\partial_z = \tfrac{1}{2}(\partial_{x_0} - e_1 \partial_{x_1}).
\]

Motivated by this decomposition, we consider functions admitting a double series expansion in the variables $z$ and $\bar{z}$ of the form
\begin{equation}\label{eq:biaxial_expansion}
F(X) = \sum_{k=0}^{\infty} \sum_{\ell=0}^{\infty} z^k \bar{z}^{\ell} \, A_{k,\ell}(\underline{y}),
\end{equation}
where the coefficient functions $A_{k,\ell} : \mathbb{R}^{m-1} \to \mathbb{R}_{0,m}$ depend only on the variables $x_2,\ldots,x_m$. Throughout this paper, such series are understood in a formal sense and questions of convergence are not considered.

The representation \eqref{eq:biaxial_expansion} is natural, since any function admitting a power series expansion in the variables $x_0$ and $x_1$ can be equivalently rewritten in terms of $z$ and $\bar{z}$, which form an alternative coordinate system for the two-dimensional component.

We now investigate under which conditions the function $F$ is monogenic, that is,
\[
\partial_X F(X) = 0.
\]
This leads to the following characterization.

\begin{theorem}\label{thm:biaxial_CK}
Let $F$ be given by \eqref{eq:biaxial_expansion}. Then $F$ is monogenic if and only if the coefficients $A_{k,\ell}$ satisfy
\begin{equation}\label{eq:CKcoefficients}
A_{k,\ell}(\underline{y}) =
\begin{cases}
\displaystyle
\frac{(k-\ell)!}{2^{2\ell} \, k! \, \ell!}
\partial_{\underline{y}}^{2\ell} A_{k-\ell,0}(\underline{y}),
& \text{if } k \ge \ell, \\[1.2em]

\displaystyle
-\frac{(\ell-k-1)!}{2^{2k+1} \, k! \, \ell!}
\partial_{\underline{y}}^{2k+1} A_{\ell-k-1,0}(\underline{y}),
& \text{if } k < \ell.
\end{cases}
\end{equation}
The coefficients $A_{k,\ell}$ are uniquely determined by the initial data $A_{k,0}$. These are obtained from the restriction of $F$ at $z=0$ together with its $z$-derivatives, more precisely,
\begin{equation}\label{eq:Ak0}
A_{k,0}(\underline{y}) = \frac{1}{k!} \partial_z^k F(X) \big|_{z=0}.
\end{equation}
\end{theorem}

\begin{proof}
Using the decomposition $\partial_X = 2\partial_{\bar z} + \partial_{\underline{y}}$, we obtain
\begin{align*}
\partial_{\bar z}\left(z^k \bar{z}^{\ell} \, A_{k,\ell}(\underline{y})\right)
&= \ell \, z^k \bar{z}^{\ell-1} \, A_{k,\ell}(\underline{y}),\\
\partial_{\underline{y}}\left(z^k \bar{z}^{\ell} \, A_{k,\ell}(\underline{y})\right)
&= \bar{z}^k z^{\ell} \, \partial_{\underline{y}}A_{k,\ell}(\underline{y}).
\end{align*}
It follows that
\[
\partial_X F(X) =
2\sum_{k=0}^{\infty} \sum_{\ell=1}^{\infty}\ell \, z^k \bar{z}^{\ell-1} \, A_{k,\ell}(\underline{y})
+
\sum_{k=0}^{\infty} \sum_{\ell=0}^{\infty}\bar{z}^k z^{\ell} \, \partial_{\underline{y}}A_{k,\ell}(\underline{y}).
\]
Hence,
\[
\partial_X F(X) =
\sum_{k=0}^{\infty} \sum_{\ell=0}^{\infty}
z^k \bar{z}^{\ell}
\left(
2(\ell+1)\, A_{k,\ell+1}(\underline{y})
+ \partial_{\underline{y}}A_{\ell,k}(\underline{y})
\right).
\]
Hence, the monogenicity condition $\partial_X F(X)=0$ is equivalent to
\[
2(\ell+1)\, A_{k,\ell+1}(\underline{y})
+ \partial_{\underline{y}}A_{\ell,k}(\underline{y}) = 0,
\qquad k,\ell \ge 0.
\]
Equivalently, this yields the recurrence relation
\[
A_{k,\ell}(\underline{y})
=
-\frac{1}{2\ell} \, \partial_{\underline{y}} A_{\ell-1,k}(\underline{y}),
\qquad k \ge 0,\ \ell \ge 1.
\]
This recurrence relation can be solved explicitly, yielding the formula \eqref{eq:CKcoefficients} for the coefficients $A_{k,\ell}$ and showing that they are uniquely determined by the initial data $A_{k,0}$. Equation \eqref{eq:Ak0} follows directly from the expansion \eqref{eq:biaxial_expansion}. This completes the proof.
\end{proof}

\begin{example}
Consider the initial conditions 
\[ A_{0,0}(\underline{y}) = x_j,\quad A_{1,0}(\underline{y}) = 1,\quad A_{k,0}(\underline{y}) = 0,\quad k \ge 2,\]
where $j\ge 2$. Applying \eqref{eq:CKcoefficients}, we obtain the monogenic function
\[F(X)=x_j + z - \frac{1}{2}\,\bar{z}\,e_j, \quad j \ge 2.\]
\end{example}

\begin{example} 
Let
\[ A_{0,0}(\underline{y}) = x_j^2,\quad A_{1,0}(\underline{y}) = x_j,\quad A_{2,0}(\underline{y}) = 1,\quad A_{k,0}(\underline{y}) = 0,\quad k \ge 3,\]
where $j\ge 2$. Using \eqref{eq:CKcoefficients}, we obtain
\[ F(X) = x_j^2 + z\,x_j - \bar z\,e_j x_j + z^2 -\frac{1}{2}\,z\bar z -\frac{1}{4}\,\bar z^{\,2}e_j. \]
\end{example}

We now fix $k \ge 0$ and determine all coefficients generated by the initial value $A_{k,0}$. The corresponding part of the series \eqref{eq:biaxial_expansion} is given by
\[
F_k(X) =
\sum_{s=0}^{\infty} z^{k+s}\bar{z}^{s} \, A_{k+s,s}(\underline{y})
+
\sum_{s=0}^{\infty} z^{s}\bar{z}^{k+s+1} \, A_{s,k+s+1}(\underline{y}).
\]

We refer to $F_k$ as the \emph{$k$-th monogenic layer} associated with the initial data $A_{k,0}$. It may be viewed as the analogue of the term $\frac{f^{(k)}(0)}{k!} z^k$ in the Taylor expansion of a holomorphic function $f(z)$.

This yields monogenic functions of the form
\[
F_k(X) =
z^{k}\sum_{s=0}^{\infty} |z|^{2s} \, B_{s}(\underline{y})
+
\bar{z}^{k+1}\sum_{s=0}^{\infty} |z|^{2s} \, C_{s}(\underline{y}).
\]

We may also determine when $F_k$ is monogenic directly in terms of the functions $B_s$ and $C_s$. Substituting this series into the monogenicity condition $\partial_X F_k = 0$ and comparing like coefficients, we obtain
\[
B_{s}(\underline{y})=-\frac{1}{2s}\,\partial_{\underline{y}}C_{s-1}(\underline{y}),\quad s\ge1,
\]
\[
C_{s-1}(\underline{y})=-\frac{1}{2(k+s)}\,\partial_{\underline{y}}B_{s-1}(\underline{y}),\quad s\ge1.
\]

Solving these recurrence relations, or equivalently applying Theorem~\ref{thm:biaxial_CK} we obtain
\[
B_{s}(\underline{y})=\frac{k!}{2^{2s}s!(k+s)!}\,\partial_{\underline{y}}^{2s}B_{0}(\underline{y}),\quad s\ge1,
\]
\[
C_{s}(\underline{y})=-\frac{k!}{2^{2s+1}s!(k+s+1)!}\,\partial_{\underline{y}}^{2s+1}B_{0}(\underline{y}),\quad s\ge0.
\]
As expected, the initial data for the $k$-th monogenic layer is given by a single function, namely $B_0$ (corresponding to $A_{k,0}$).

Substituting these expressions into $F_k$, we obtain
\[
F_k(X) =
k!\left(
z^{k}\sum_{s=0}^{\infty} \frac{|z|^{2s}}{2^{2s}s!(k+s)!}\,\partial_{\underline{y}}^{2s}B_{0}(\underline{y})
-
\bar{z}^{k+1}\sum_{s=0}^{\infty} \frac{|z|^{2s}}{2^{2s+1}s!(k+s+1)!}\,\partial_{\underline{y}}^{2s+1}B_{0}(\underline{y})
\right).
\]

Recalling that $B_0(\underline{y}) = A_{k,0}(\underline{y})$ and using the identity \eqref{eq:Ak0}, we rewrite $F_k$ as
\begin{equation*}
F_k(X) = \mathrm{CK}_{z,k}\left[g_k(\underline{y})\right](X),
\end{equation*}
where the operator $\mathrm{CK}_{z,k}$ is given by
\begin{equation}\label{eq:CKzk}
\begin{aligned}
\mathrm{CK}_{z,k}\left[g_k(\underline{y})\right](X) &=
z^{k}\sum_{s=0}^{\infty} \frac{|z|^{2s}}{2^{2s}s!(k+s)!}\,\partial_{\underline{y}}^{2s}g_k(\underline{y})\\
&\qquad -\bar{z}^{k+1}\sum_{s=0}^{\infty} \frac{|z|^{2s}}{2^{2s+1}s!(k+s+1)!}\,\partial_{\underline{y}}^{2s+1}g_k(\underline{y}).
\end{aligned}
\end{equation}
and
\[
g_k(\underline{y}) = \partial_z^k F(X)\big|_{z=0}.
\]
The explicit expression for $F_k$ suggests a connection with special functions. Indeed, $F_k$ can be rewritten in terms of the modified Bessel function of the first kind, defined by
\[
I_k(x)=\sum_{s=0}^{\infty}\frac{1}{s!(k+s)!}\left(\frac{x}{2}\right)^{2s+k}.
\]
This observation allows us to express $F_k$ in a more compact form using functional calculus. 
\begin{equation}\label{eq:CKzkBessel}
\begin{aligned}
\mathrm{CK}_{z,k}\left[g_k(\underline{y})\right](X) &=
z^{k}\left(\frac{2}{|z|\partial_{\underline{y}}}\right)^k
I_k\bigl(|z|\partial_{\underline{y}}\bigr) g_k(\underline{y})\\
&\qquad -\frac{\bar{z}^{k+1}}{|z|}
\left(\frac{2}{|z|\partial_{\underline{y}}}\right)^k
I_{k+1}\bigl(|z|\partial_{\underline{y}}\bigr) g_k(\underline{y}).
\end{aligned}
\end{equation}

\begin{corollary}
Let $\{g_k(\underline{y})\}_{k \ge 0}$ be a sequence of real-analytic functions. Then there exists a unique monogenic function $F$ satisfying
\[
\partial_z^k F(X)\big|_{z=0} = g_k(\underline{y}), \qquad k \ge 0,
\]
given by
\[
F(X) = \sum_{k=0}^{\infty} \mathrm{CK}_{z,k}\left[g_k(\underline{y})\right](X),
\]
where $\mathrm{CK}_{z,k}[g_k(\underline{y})]$ is given by \eqref{eq:CKzk}, and can equivalently be written in the Bessel-type form \eqref{eq:CKzkBessel}.
\end{corollary}

This construction extends the classical CK extension from the hyperplane $\{X\in\mathbb{R}^{m+1}: x_0=0\}$ to the codimension-2 subspace $\{X\in\mathbb{R}^{m+1}: z=0\}$. While the classical setting prescribes a single initial function on the $m$-dimensional hyperplane, the present framework considers data on the lower-dimensional subspace of dimension $m-1$. Consequently, the initial data are no longer given by a single function, but by a sequence of functions $\{g_k\}_{k \ge 0}$ compensating for the reduction in dimension.

\begin{remark}
The operator $\mathrm{CK}_{z,k}$ generates monogenic functions of the form
\[z^{k} B(|z|, \underline{y}) + \bar{z}^{k+1}C(|z|, \underline{y}).\]
One particular yet interesting case arises when $B$ and $C$ are independent of $|z|$. In this case, monogenicity is equivalent to 
\[\partial_{\underline{y}}C(\underline{y})=0, \qquad C(\underline{y}) = -\frac{1}{2(k+1)} \partial_{\underline{y}}B(\underline{y}). \]
Substituting the second equation into the first gives
\[\partial_{\underline{y}}^{2}B(\underline{y}) = -\Delta_{\underline{y}}B(\underline{y}) = 0. \]
Hence, $B(\underline{y})$ is harmonic, and the corresponding monogenic function takes the form
\[z^{k}B(\underline{y}) - \frac{\bar{z}^{k+1}}{2(k+1)}\,\partial_{\underline{y}}B(\underline{y}). \]
These functions are precisely the power steering monogenic functions introduced in \cite{DPFSSt}. 
\end{remark}

In the remainder of this paper, we take a closer look at two families of monogenic functions generated by $\mathrm{CK}_{z,k}$.

\section{Biaxial monogenic functions}

The action of the operators $\mathrm{CK}_{z,k}$ on the functions $g_k(\underline{y})$ can be described using the fact that each $g_k(\underline{y})$ is real analytic in $\mathbb{R}^{m-1}$ and thus admits a convergent expansion of the form
\[
g_k(\underline{y}) = \sum_{d=0}^{\infty} H_d(\underline{y}),
\]
where each $H_d(\underline{y})$ is a homogeneous polynomial of degree $d$.

Using the Fischer decomposition in Clifford analysis (see \cite{DSomS}), each homogeneous component $H_d(\underline{y})$ can be written as a finite sum of terms of the form
\[
\underline{y}^j M_{d-j}(\underline{y}), \qquad j=0,\dots,d,
\]
where $M_{d-j}(\underline{y})$ are homogeneous monogenic polynomials in $\mathbb{R}^{m-1}$ of degree $d-j$.

Since $\mathrm{CK}_{z,k}$ is additive, it suffices to determine how the operators act on polynomials of the form $\underline{y}^j M_{n}(\underline{y})$.

Depending on the parity of the exponent $j$, and as a consequence of identity \eqref{LeibnizRuleTwo}, one obtains monogenic functions of the form
\begin{equation}\label{eq:biaxial_even}
\left(z^k A(u,v) + \bar{z}^{k+1}\underline{y}\, B(u,v)\right) M_n(\underline{y}),\qquad \text{if $j$ is even},
\end{equation}
or
\begin{equation}\label{eq:biaxial_odd}
\left(z^k \underline{y}\,C(u,v) + \bar{z}^{k+1} D(u,v)\right) M_n(\underline{y}),\qquad \text{if $j$ is odd},
\end{equation}
where
\[
u = \frac{|z|^2}{2}, \qquad v = \frac{|\underline{y}|^2}{2},
\]
and $A, B, C$ and $D$ are $\mathbb{R}$-valued continuously differentiable functions. 

\begin{theorem}
Let $M_n(\underline{y})$ be a homogeneous monogenic polynomial in $\mathbb{R}^{m-1}$ of degree $n$. Then the function \eqref{eq:biaxial_even} is monogenic if and only if $A$ and $B$ satisfy
\begin{equation}\label{eq:system-AB}
\begin{aligned}
\partial_u A - 2v\,\partial_v B &= (2n + m - 1)B, \\[4pt]
\partial_v A + 2u\, \partial_u B &= -2(k+1)B,
\end{aligned}
\end{equation}
and the function \eqref{eq:biaxial_odd} is monogenic if and only if $C$ and $D$ satisfy
\begin{equation}\label{eq:system-CD}
\begin{aligned}
\partial_u C + \partial_v D &= 0, \\[4pt]
v\, \partial_v C - u\, \partial_u D &= -\dfrac{2n + m - 1}{2}\, C + (k+1)D.
\end{aligned}
\end{equation}
\end{theorem}

\begin{proof}
We denote by $G(X)$ and $H(X)$ the functions defined in \eqref{eq:biaxial_even} and \eqref{eq:biaxial_odd}, respectively. Using identities \eqref{LeibnizRuleOne} and \eqref{LeibnizRuleTwo}, we obtain

\begin{align*}
2\partial_{\bar z}G &= \Bigl(z^{k+1}\partial_uA+2(k+1)\bar{z}^{k}\,\underline{y}\,B + \bar{z}^{k+1}z\,\underline{y}\,\partial_uB \Bigr) M_n\\
\partial_{\underline{y}}G &= \Bigl(\bar{z}^{k}\,\underline{y}\,\partial_vA+z^{k+1}\underline{y}^2\, \partial_vB-(2n + m - 1)z^{k+1}B  \Bigr) M_n
\end{align*}
and
\begin{align*}
2\partial_{\bar z}H &= \Bigl(z^{k+1}\underline{y}\,\partial_uC+2(k+1)\bar{z}^{k}D + \bar{z}^{k+1}z\,\partial_uD \Bigr) M_n\\
\partial_{\underline{y}}H &= \Bigl(\bar{z}^{k}\underline{y}^2\, \partial_vC-(2n + m - 1)\bar{z}^{k}C + z^{k+1}\,\underline{y}\,\partial_vD\Bigr) M_n.
\end{align*}

Therefore,
\[
\begin{aligned}
\partial_X G = {} & \left(
z^{k+1}\bigl(\partial_u A - 2v\, \partial_v B - (2n + m - 1)B\bigr)
\right. \\
& \left. \qquad + \bar{z}^{k}\,\underline{y}\,
\bigl(\partial_v A + 2u\, \partial_u B + 2(k+1)B\bigr)
\right) M_n
\end{aligned}
\]
and
\[
\begin{aligned}
\partial_X H = {} & \left(
z^{k+1}\,\underline{y}\,\bigl(\partial_u C + \partial_vD\bigr)
\right. \\
& \left. \qquad + \bar{z}^{k}
\bigl(2u\,\partial_uD- 2v\,\partial_v C -(2n + m - 1)C + 2(k+1)D\bigr)
\right) M_n.
\end{aligned}
\]
Hence, the result follows from the preceding two identities.
\end{proof}

We solve the systems \eqref{eq:system-AB} and \eqref{eq:system-CD} by means of a power series method. To this end, we seek solutions of the form 
\[ A(u,v)=\sum_{s=0}^{\infty} u^s a_s(v),\qquad B(u,v)=\sum_{s=0}^{\infty} u^s b_s(v), \]
and
\[ C(u,v)=\sum_{s=0}^{\infty} u^s c_s(v),\qquad D(u,v)=\sum_{s=0}^{\infty} u^s d_s(v), \]
where $a_s$, $b_s$, $c_s$ and $d_s$ are functions of $v$.

Substituting the series expansions for $A$ and $B$ into \eqref{eq:system-AB} and identifying the coefficients of equal powers of $u$, we obtain, 
\begin{align*}
(s+1)a_{s+1}(v)-2v\,b_s'(v) &=(2n + m - 1)b_s(v), \\
a_s'(v)+2s\,b_s(v) &=-2(k+1)b_s(v), 
\end{align*}
for $s\ge0$. The second relation can be rewritten as
\begin{equation}\label{eq:bseries}
b_s(v) = -\frac{a_s'(v)}{2(k+s+1)}. 
\end{equation}
Substituting this expression into the first relation yields 
\begin{equation}\label{eq:aseries}
a_{s+1}(v) = -\frac{1}{(s+1)(k+s+1)} \left( v\,a_s''(v) +\frac{2n+m-1}{2}\,a_s'(v) \right). 
\end{equation}

Given the initial function $a_0(v)$, relation \eqref{eq:aseries} uniquely determines the sequence of functions $\{a_s(v)\}_{s\ge0}$, while the sequence $\{b_s(v)\}_{s\ge0}$ is obtained from \eqref{eq:bseries}. Since $A(0,v)=a_0(v)$, the functions $A(u,v)$ and $B(u,v)$ are therefore uniquely determined by the restriction of $A$ to $u=0$.

Proceeding as above, substitution of the series expansions for $C$ and $D$ into \eqref{eq:system-CD} and comparison of the coefficients of equal powers of $u$ yield

\begin{align*}
(s+1)c_{s+1}(v) + d_s'(v) &=0, \\ 
v\,c_s'(v)-s\,d_s(v) &=-\frac{2n + m - 1}{2}\,c_s(v)+(k+1)d_s(v), 
\end{align*}
for $s\ge0$. The second relation can be rewritten as 
\begin{equation}\label{eq:dseries}
d_s(v) = \frac{1}{k+s+1} \left( v\,c_s'(v) +\frac{2n+m-1}{2}\,c_s(v) \right). 
\end{equation}
Substituting the above expression into the first relation, we obtain 
\begin{equation}\label{eq:cseries}
c_{s+1}(v) = -\frac{1}{(s+1)(k+s+1)} \left( v\,c_s''(v) +\frac{2n+m+1}{2}\,c_s'(v) \right). 
\end{equation}

Observe that the sequences of functions $\{c_s(v)\}_{s\ge0}$ and $\{d_s(v)\}_{s\ge0}$ are uniquely determined by the initial function $c_0(v)$. Since $C(0,v)=c_0(v)$, it follows that the functions $C(u,v)$ and $D(u,v)$ are uniquely determined by the restriction of $C$ to $u=0$.

\begin{example}
Assume that
\[a_0(v)=v^p, \quad p\in\mathbb{R}.\]
Then it follows from \eqref{eq:aseries} that 
\[a_s(v) = (-1)^s \frac{k!}{s!(k+s)!} (p)_s \left(p+n+\frac{m-3}{2}\right)_s v^{p-s}, \]
where 
\[(x)_s=x(x-1)\cdots(x-(s-1))\] 
denotes the falling factorial. Using \eqref{eq:bseries}, we obtain
\[b_s(v) = \frac{(-1)^{s+1}k!}{2\,s!(k+s+1)!} (p)_{s+1} \left(p+n+\frac{m-3}{2}\right)_s v^{p-s-1}. \]
Consequently, 
\[A(u,v) = v^p \sum_{s=0}^{\infty}\frac{k!}{s!(k+s)!} (p)_s \left(p+n+\frac{m-3}{2}\right)_s \left(-\frac{u}{v}\right)^s, \]
Using
\[(x)_s=(-1)^s(-x)^{(s)}, \]
where
\[ (x)^{(s)}=x(x+1)\cdots(x+(s-1))\]
denotes the rising factorial, we obtain 
\[A(u,v) = v^p \sum_{s=0}^{\infty}\frac{1}{s!(k+1)^{(s)}} (-p)^{(s)} \left(-p-n-\frac{m-3}{2}\right)^{(s)} \left(-\frac{u}{v}\right)^s.\]
The latter series is recognized as a Gauss hypergeometric series,
\[ {}_2F_1(a,b;c;z) = \sum_{s=0}^{\infty} \frac{(a)^{(s)}(b)^{(s)}}{(c)^{(s)}} \frac{z^s}{s!}. \]
Hence,
\[A(u,v) = v^p\, {}_2F_1 \!\left( -p,\, -p-n-\frac{m-3}{2};\, k+1;\, -\frac{u}{v} \right). \]
Similarly, one obtains
\[B(u,v) = -\frac{p}{2(k+1)} v^{p-1} \,{}_2F_1 \!\left( 1-p,\, -p-n-\frac{m-3}{2};\, k+2;\, -\frac{u}{v} \right). \]
In particular, if $p\in\mathbb{N}$, then both hypergeometric series terminate, and hence $A(u,v)$ and $B(u,v)$ reduce to polynomials in the variables $u$ and $v$.
\end{example}

\begin{example}
Suppose that
\[c_0(v)=v^p, \quad p\in\mathbb{R}\]
Relations \eqref{eq:dseries} and \eqref{eq:cseries} yield
\begin{align*}
c_s(v) &= (-1)^s \frac{k!}{s!(k+s)!} (p)_s \left(p+n+\frac{m-1}{2}\right)_s v^{p-s},\\
d_s(v) &= (-1)^s \frac{k!}{s!(k+s+1)!} (p)_s \left(p+n+\frac{m-1}{2}\right)_{s+1} v^{p-s}.
\end{align*}
Hence, the functions $C(u,v)$ and $D(u,v)$ admit the hypergeometric representations
\[C(u,v) = v^p\, {}_2F_1 \!\left( -p,\, -p-n-\frac{m-1}{2};\, k+1;\, -\frac{u}{v} \right)\]
and
\[D(u,v) = \frac{p+n+\frac{m-1}{2}}{k+1}\, v^p\, {}_2F_1 \!\left( -p,\, -p-n-\frac{m-3}{2};\, k+2;\, -\frac{u}{v} \right).\]
\end{example}

\section{Biaxial monogenic plane waves}

Another way to characterize the operators $\mathrm{CK}_{z,k}$ is by studying their action on polynomials. It suffices to determine their effect on the monomial basis
\[\underline{y}^{\underline{\alpha}} = x_2^{\alpha_2}\cdots x_m^{\alpha_m}, \qquad \underline{\alpha}=(\alpha_2,\ldots,\alpha_m)\in\mathbb{N}^{m-1}, \qquad |\underline{\alpha}|=p\in\mathbb{N}. \]
For this purpose, we make use of the monomial reproducing kernel
\[K_p(\underline{y},\underline{t}) = \frac{\langle \underline{y},\underline{t}\rangle^p}{p!}, \quad \underline{t} = \sum_{j=2}^m t_j e_j,\]
which satisfies 
\[\partial_{\underline{t}}^{\underline{\alpha}}\,K_p(\underline{y},\underline{t}) = \underline{y}^{\underline{\alpha}}, \]
where 
\[\partial_{\underline{t}}^{\underline{\alpha}} = \partial_{t_2}^{\alpha_2}\cdots \partial_{t_m}^{\alpha_m}. \]
A straightforward computation yields
\[\mathrm{CK}_{z,k}\left[K_p(\underline{y},\underline{t})\right](X)=W_{p,k}(z,\underline{y},\underline{t}),\]
with
\[W_{p,k}(z,\underline{y},\underline{t})=z^{k} W_{p,k}^{(1)}(z,\underline{y},\underline{t}) + \bar z^{\,k+1}\,\underline{t}\, W_{p,k}^{(2)}(z,\underline{y},\underline{t})\]
and
\begin{align*}
W_{p,k}^{(1)}(z,\underline{y},\underline{t}) &= \sum_{s=0}^{\lfloor p/2\rfloor} \frac{(-1)^s\left(|z|\,|\underline{t}|\right)^{2s} \langle \underline{y},\underline{t}\rangle^{p-2s}} {2^{2s}\,s!\,(k+s)!\,(p-2s)!},\\
W_{p,k}^{(2)}(z,\underline{y},\underline{t}) &= \sum_{s=0}^{\left\lfloor (p-1)/2\right\rfloor} \frac{(-1)^{s+1}\left(|z|\,|\underline{t}|\right)^{2s} \langle \underline{y},\underline{t}\rangle^{p-2s-1}} {2^{2s+1}\,s!\,(k+s+1)!\,(p-2s-1)!},
\end{align*}
where $\lfloor x \rfloor$ denotes the floor function.

Combining the above observations, we arrive at the following theorem.

\begin{theorem}
Let 
\[H_p(\underline{y}) = \sum_{|\underline{\alpha}|=p} \underline{y}^{\underline{\alpha}}\,a_{\underline{\alpha}}, \quad a_{\underline{\alpha}}\in\mathbb{R}_{0,m}, \]
be a homogeneous polynomial in $\mathbb{R}^{m-1}$ of degree $p$. Then
\[\mathrm{CK}_{z,k}\left[H_p(\underline{y})\right](X)=\sum_{|\underline{\alpha}|=p}\left(\partial_{\underline{t}}^{\underline{\alpha}}\,W_{p,k}(z,\underline{y},\underline{t})\right)a_{\underline{\alpha}}.\]
\end{theorem}

More generally, the operators $\mathrm{CK}_{z,k}$ may be applied to arbitrary $\mathbb{R}$-valued real-analytic functions of the scalar product $\langle \underline{y}, \underline{t} \rangle$.

\begin{example}
Consider the initial function $e^{\langle \underline{y},\underline{t}\rangle}$. Then
\[\mathrm{CK}_{z,k}\left[e^{\langle \underline{y},\underline{t}\rangle}\right](X) =e^{\langle \underline{y},\underline{t}\rangle}\left(z^{k}\sum_{s=0}^{\infty} \frac{(-1)^s\left(|z|\,|\underline{t}|\right)^{2s}}{2^{2s}s!(k+s)!}-\bar{z}^{k+1}\,\underline{t}\, \sum_{s=0}^{\infty} \frac{(-1)^{s}\left(|z|\,|\underline{t}|\right)^{2s}}{2^{2s+1}s!(k+s+1)!}\right).\]
By means of the Bessel functions of the first kind
\[J_k(x) = \sum_{s=0}^{\infty} \frac{(-1)^s}{s!(k+s)!} \left(\frac{x}{2}\right)^{2s+k},\]
we may rewrite the preceding function as
\[\mathrm{CK}_{z,k}\left[e^{\langle \underline{y},\underline{t}\rangle}\right](X) =\left(\frac{2}{|z|\,|\underline{t}|}\right)^ke^{\langle \underline{y},\underline{t}\rangle}\left(z^{k}J_k\left(|z|\,|\underline{t}|\right)-\bar{z}^{k+1}\,\frac{\underline{t}}{|z|\,|\underline{t}|}\,J_{k+1}\left(|z|\,|\underline{t}|\right)\right).\]
\end{example}

The explicit forms of 
\[\mathrm{CK}_{z,k}\left[\langle \underline{y},\underline{t}\rangle^p\right](X) \quad \text{and} \quad \mathrm{CK}_{z,k}\left[e^{\langle \underline{y},\underline{t}\rangle}\right](X)\]
naturally motivate the introduction of the following class of biaxial monogenic plane wave functions:
\begin{equation}\label{MonPlaneWaves}
G(z,\underline{y},\underline{t})=z^{k}A(u,v) + \bar{z}^{k+1}\,\underline{t}\, B(u,v),
\end{equation}
where
\[u=\frac{\left(|z|\,|\underline{t}|\right)^2}{2},\qquad v=\langle \underline{y},\underline{t}\rangle,\]
and $A$ and $B$ are $\mathbb{R}$-valued continuously differentiable functions.

\begin{theorem}
The function given by \eqref{MonPlaneWaves} is monogenic if and only if the functions $A$ and $B$ satisfy
\begin{equation}\label{PWsystem}
\begin{aligned}
\partial_uA - \partial_vB &= 0\\[4pt]
\partial_v A + 2u\, \partial_uB &= -2(k+1)B.
\end{aligned}
\end{equation}
\end{theorem}

\begin{proof}
A direct computation yields
\begin{align*}
2\partial_{\bar z}G &=z^{k}\left(\partial_uA\right)z|\underline{t}|^2+2(k+1)\bar{z}^{k}\,\underline{t}\,B+\bar{z}^{k+1}\left(\partial_uB\right)z|\underline{t}|^2\underline{t}\\
2\partial_{\bar z}G &= z^{k+1}|\underline{t}|^2\partial_uA+2(k+1)\bar{z}^{k}\,\underline{t}\,B+2\bar{z}^{k}\,\underline{t}\,u\partial_uB
\end{align*}
and
\begin{align*}
\partial_{\underline{y}}G&=\left(\partial_vA\right)\underline{t}\,z^k+\left(\partial_vB\right)\underline{t}\,\bar{z}^{k+1}\,\underline{t}\\
\partial_{\underline{y}}G &=\bar{z}^k\,\underline{t}\,\partial_vA-{z}^{k+1}|\underline{t}|^2\partial_vB.
\end{align*}
Therefore
\[
\partial_X G = z^{k+1}|\underline{t}|^2\big(\partial_uA - \partial_vB\big)+\bar{z}^{k}\,\underline{t}\big(\partial_vA+2u\partial_uB+2(k+1)B\big),
\]
from which the result follows immediately.
\end{proof}

As in the previous section, we employ the power series method to solve system \eqref{PWsystem}. We seek solutions of the form
\[A(u,v)=\sum_{s=0}^{\infty} u^s a_s(v),\qquad B(u,v)=\sum_{s=0}^{\infty} u^s b_s(v), \]
where the coefficients $a_s$ and $b_s$ depend only on $v$.

Substituting these series into \eqref{PWsystem} and comparing coefficients of equal powers of $u$, we arrive at 
\[(s+1)a_{s+1}(v)=b_s'(v), \qquad a_s'(v)=-2(k+s+1)b_s(v), \qquad s\ge 0. \]
Eliminating $b_s(v)$, we obtain the recurrence relation 
\[a_{s+1}(v) = -\frac{a_s''(v)}{2(s+1)(k+s+1)}, \qquad s\ge 0. \]
Iterating gives
\[a_s(v) = \frac{(-1)^sk!}{2^s s!(k+s)!}\, a_0^{(2s)}(v), \]
and consequently
\[b_s(v) = \frac{(-1)^{s+1}k!}{2^{s+1}s!(k+s+1)!}\, a_0^{(2s+1)}(v). \]
Clearly, the functions $A(u,v)$ and $B(u,v)$ are uniquely determined by the initial function $a_0(v)=A(0,v)$.

\end{document}